\documentclass[12pt]{amsart}
\usepackage[margin=1.2in]{geometry}

\usepackage{amsfonts}
\usepackage{amssymb}
\usepackage{amsmath}
\usepackage{amsthm}
\usepackage{mathtools}
\usepackage{xcolor}
\usepackage{hyperref}

\theoremstyle{plain}
\newtheorem{theorem}{Theorem}[section]
\newtheorem{lemma}[theorem]{Lemma}
\newtheorem{proposition}[theorem]{Proposition}

\newtheorem{conjecture}[theorem]{Conjecture}

\theoremstyle{definition}
\newtheorem{definition}[theorem]{Definition}

\newtheorem{remark}[theorem]{Remark}

\numberwithin{equation}{section}

\newcommand{\N}{\mathbb{N}}

\newcommand{\R}{\mathbb{R}}
\newcommand{\C}{\mathbb{C}}

\title[Strong factorization theorem for smooth vectors]{Strong factorization theorem for smooth vectors of exponential solvable Lie group representations}

\author[S. Chaves]{Santiago Chaves}
\address{S. Chaves\\ Centro de Investigaci\'{o}n en Matemática Pura y Aplicada\\ Escuela de Matem\'{a}tica\\Universidad de Costa Rica\\Montes de Oca\\ San Jos\'{e}\\
Costa Rica}
\email{anthony.chavesaguilar@ucr.ac.cr}
\thanks{This work was supported by the University of Costa Rica through the  project ``Avances en la teoría de factorización de vectores suaves y analíticos de representaciones de grupos de Lie" via Fondo Semilla 2024 from Vicerrector\'{i}a de Investigaci\'{o}n.}

\thanks {S. Chaves, A Hernández, and R. Zamora acknowledge Escuela de Matemática and Centro de Investigación en Matemática Pura y Aplicada at Universidad de Costa Rica for their support via project: C4450.}

\author[A. Debrouwere]{Andreas Debrouwere}
\address{A. Debrouwere, Department of Mathematics and Data Science \\ Vrije Universiteit Brussel, Belgium\\ Pleinlaan 2 \\ 1050 Brussels \\ Belgium}
\email{andreas.debrouwere@vub.be}
\thanks{A. Debrouwere acknowledges support from the Research Foundation-Flanders through the grant G0A1O26N}

\author[A. Hern\'{a}ndez]{Alberto Hern\'{a}ndez Alvarado}
\address{A. Hern\'{a}ndez\\ Centro de Investigaci\'{o}n en Matemática Pura y Aplicada\\ Escuela de Matem\'{a}tica\\Universidad de Costa Rica\\Montes de Oca\\ San Jos\'{e}\\
Costa Rica}
\email{albertojose.hernandez@ucr.ac.cr}

\author[J. Vindas]{Jasson Vindas}
\address{J. Vindas, Department of Mathematics: Analysis, Logic and Discrete Mathematics\\ Ghent University\\ Krijgslaan 297\\ 9000 Ghent\\ Belgium}
\email{jasson.vindas@UGent.be}
\thanks {J.~Vindas acknowledges support from the Research Foundation-Flanders through the grants G067621N and K801226N}

\author[R. Zamora]{Rafael Zamora}
\address{R. Zamora\\ Centro de Investigaci\'{o}n en Matemática Pura y Aplicada\\ Escuela de Matem\'{a}tica\\Universidad de Costa Rica\\Montes de Oca\\ San Jos\'{e}\\
Costa Rica}
\email{rafael.zamora\_c@ucr.ac.cr}


\begin{document}

\begin{abstract}

We establish new strong factorization properties for the smooth vectors of representations of exponential solvable Lie groups on Fr\'{e}chet spaces. 
    In particular, our results improve upon the Dixmier-Malliavin factorization theorem for simply connected nilpotent Lie groups.
\end{abstract}

\maketitle

\section{Introduction}
In their seminal paper \cite{D-M}, Dixmier and Malliavin established a number of important factorization properties for smooth vectors associated with Lie group representations. Their work was motivated by Ehrenpreis' factorization problem \cite{ehrenpreis} for the Schwartz space of compactly supported smooth functions on $\mathbb{R}^{n}$. Ehrenpreis' question turned out to be a fascinating challenging problem \cite{D-M,rst}, whose final solution was only found until 1998 by Yulmukhametov \cite{yul}.

 Let $\pi$ be a representation of a (real) Lie group $G$ on a Fréchet space $E$. 
The space of smooth vectors $E^{\infty}=E^{\infty}_{\pi}$ of $\pi$ consists of all those vectors $v\in E$ whose orbit map $g \mapsto \pi(g)v$ is a smooth $E$-valued function on $G$; it carries a canonical Fr\'{e}chet space topology (see Subsection \ref{subsect representations}). We also consider the space $\mathcal{D}(G)$ of compactly supported smooth functions on the group.
The representation $\pi$ induces a natural action $\Pi$ of $\mathcal{D}(G)$ on the Fr\'{e}chet space $E^{\infty}$, given by 
\begin{equation}
    \label{equ:action}
    \Pi(\chi)v 
    = \int _{G} \chi(g) \pi(g) v \: d{g},
\end{equation}
for all $\chi\in \mathcal{D}(G)$ and $v\in E^{\infty}$, where here $dg$ is a left-invariant Haar measure. Dixmier and Malliavin showed \cite{D-M} that $E^{\infty}$ always has the so-called weak factorization property with respect to the convolution algebra $\mathcal{D}(G)$, that is, $E^{\infty} =  \operatorname*{span}(\Pi(\mathcal{D}(G)) E^{\infty}).$  
 
When the representation $\pi$ satisfies additional moderate growth conditions, the definition \eqref{equ:action} of the action $\Pi$ still makes sense for larger smooth function algebras than just $\mathcal{D}(G)$. This is the case for example if $\pi$ is a bounded representation of a nilpotent Lie group on a Banach space (namely, $\sup_{g\in G} \|\pi(g)\|_{L(E,E)}<\infty$) and $\chi$ belongs to the Schwartz space $\mathcal{S}(G)$. The Dixmier-Malliavin strong factorization theorem for simply connected nilpotent Lie groups is the next statement, which generalizes a result of Howe \cite{H} by removing the irreducibility assumption on the representation.

\begin{theorem}[{\cite[7.4 Th\'{e}or\`{e}me]{D-M}}]\label{D-M theorem} 
Let $\pi$ be a unitary representation of a simply connected nilpotent Lie group $G$ on a Hilbert space $E$. Then, its space of smooth vectors $E^{\infty}$ has the strong factorization property with respect to the Schwartz space $\mathcal{S}(G)$, namely, $E^{\infty}=\Pi(\mathcal{S}(G)) E^{\infty}=\{\Pi(\chi)v: \: \chi\in\mathcal{S}(G) \mbox{ and }v\in E^{\infty}\}$. 
\end{theorem}

The aim of this article is to extend and improve Theorem \ref{D-M theorem}. Our main result is the following strong bounded factorization theorem for exponential solvable Lie group representations. We refer to Subsection \ref{subsec exp decreasing} below for the precise definition of the algebra of exponentially rapidly decreasing smooth functions $\mathcal{K}(G)$ on $G$, and to Subsection \ref{subsect representations} for that of representations of exponential type on Fr\'{e}chet spaces.

\begin{theorem}
\label{th: fth}
Let $\pi$ be a representation of exponential type of an exponential solvable Lie group $G$ on a Fr\'{e}chet space $E$. Then, given a bounded subset $B \subset E^{\infty}$, one can always find $\chi\in \mathcal{K}(G)$ and another bounded subset $B'\subset E^{\infty}$ such that $B=\Pi(\chi)(B')$. In particular, $E^{\infty}=\Pi(\mathcal{K}(G))E^{\infty}.$
\end{theorem}

Besides holding for broader classes of representations on more general Lie groups and yielding simultaneous strong factorization for elements in bounded sets, Theorem \ref{th: fth} also significantly improves upon  Theorem \ref{D-M theorem} in the particular case of simply connected nilpotent Lie groups $G$ by employing the much finer convolution algebra $\mathcal{K}(G)$, rather than the larger Schwartz space $\mathcal{S}(G)$ occurring in the Dixmier-Malliavin strong factorization theorem.

We will divide the proof of Theorem \ref{th: fth} into several intermediate steps. In Section \ref{sec: factorization real line}, we show the result for the special case of the real line $G=\mathbb{R}$. The general case of exponential solvable Lie groups will then be reduced to that of the real line in Section \ref{sec factorization exponential} by combining an iteration method due to Dixmier and Malliavin with results connecting $\mathcal{K}(G)$ and the convolution algebra $\mathcal{K}(\mathfrak{g})$ on the Lie algebra $\mathfrak{g}$ via special exponential coordinates of the second kind. The latter tools shall be established in Section \ref{sec exp coordinates}. 

We end this introduction by mentioning that it would be interesting to show Theorem \ref{th: fth} for other classes of Lie groups. In fact, we actually conjecture that the strong bounded factorization property holds true for any  connected Lie group:

\begin{conjecture}
\label{conjecture smooth} The space of smooth vectors $E^{\infty}$ associated with a representation of exponential type of a connected Lie group $G$ on a Fr\'{e}chet space has the strong bounded factorization property with respect to  $\mathcal{K}(G)$.
\end{conjecture}

Conjecture \ref{conjecture smooth} is the smooth counterpart of another conjecture due\footnote{In \cite{G-K-L}, the authors were able to prove that the analog of the Dixmier-Malliavin \emph{weak} factorization theorem holds for analytic vectors.} to  Gimperlein, Kr\"{o}tz, and Lienau for analytic vectors, which we state here in a slightly strengthened form. See \cite{G-K-L} for the definition of the algebra $\mathcal{A}(G)$ of exponentially rapidly decreasing analytic functions on $G$.

\begin{conjecture}[{\cite[Conjecture 6.4]{G-K-L}}]
\label{conjecture analytic} The space of analytic vectors $E^{\omega}$ associated with a representation of exponential type of a connected Lie group $G$ on a Fr\'{e}chet space has the strong bounded factorization property with respect to  $\mathcal{A}(G)$.
\end{conjecture}

Conjecture \ref{conjecture smooth} is known to hold true for Lie groups whose Lie algebra admits a basis of toroidal elements as explained in \cite[Section 4]{D-M} (for instance, compact Lie groups are of this kind), and Theorem \ref{th: fth} proves it for exponential solvable Lie groups. On the other hand, Conjecture \ref{conjecture analytic} has been established for $G=\mathbb{R}^{n}$ \cite{DPV21} and for compact Lie groups \cite{D-H-V25}.

\section{Preliminaries}
\label{sec: preliminaries}

Throughout the article $E$ stands for a Fr\'{e}chet space and $G$ for an exponential solvable Lie group, whose identity element is denoted by $e\in G$. We recall \cite{F-LBook} that by definition $G$ is simply connected, its Lie algebra $\mathfrak{g}$ is solvable, and the exponential function $\exp: \mathfrak{g}\to G$ is a global diffeomorphism. We write $\mathfrak{D}(G)$ for the algebra of all left-invariant differential operators on $G$.
We regard the elements of the Lie algebra of $G$ as left-invariant vector fields, so that $\mathfrak{g}\subset\mathfrak{D}(G)$. We denote by  $\iota:G\to G$ the inversion map, namely, $\iota(g)=g^{-1}$. Observe then that $\iota_{\ast}(\mathfrak{D}(G))$ is the algebra of right-invariant differential operators. We fix a left-invariant Haar measure $dg$ on $G$. We write  $\operatorname*{csn}(E)$ for the family of all continuous seminorms on $E$.

\subsection{Exponentially rapidly decreasing functions on a Lie group}\label{subsec exp decreasing} We fix a left-invariant Riemann metric on $G$ and denote its associated distance function as $d: G\times G\to [0,\infty)$. The latter metric is also left-invariant: 
$$d(g_3g_1,g_3g_2)=d(g_1,g_2), \qquad \mbox{for any } g_1, g_2, g_3\in G.
$$
We then use $d$ to define the function $|\cdot|_{G}= G\to [0,\infty)$ as the distance of an element to $e$, namely, 
$$|g|_G:=d(e,g), \qquad g\in G.$$
One readily verifies \cite{G60} that $|g|_G$  satisfies the inequality 
\begin{equation}
\label{eq submult} 
|g h|_G\leq |g|_G+|h|_G, \qquad \mbox{for any } g,h\in G.
\end{equation}
Furthermore, obviously, $|g |_G=|g^{-1}|_G$.

 The space of exponentially rapidly decreasing smooth functions on $G$ is defined as 
$$\mathcal{K}(G)=\Big\{\varphi\in C^{\infty}(G):\: (\forall \lambda>0)(\forall L\in\mathfrak{D}(G))\ \sup_{g\in G} e^{\lambda|g|_G}(|(L\varphi)(g)|+ |((\iota_{\ast}L)\varphi)(g)|)<\infty\Big\}.$$
(This function space coincides with the Schwartz space with respect to the maximal scale structure on the group in the terminology from \cite{B-K}.) 
Note that $\mathcal{K}(G)$ can be provided with a natural Fr\'{e}chet space structure. In fact, if $\{L_1,L_2,L_3,L_4,\dots\}$ is a basis for $\mathfrak{D}(G)$, its canonical Fr\'{e}chet space topology is generated by the countable family of seminorms
$$
\sup_{g\in G} e^{j|g|_G}(|(L_k\varphi)(g)|+ |((\iota_{\ast}L_k)\varphi)(g)|), \qquad j,k\in\mathbb{N}.
$$
The space $\mathcal{K}(G)$ is invariant under both left- and right-translations and it is a Fr\'{e}chet algebra with respect to both left- and right-convolution. 

As $\mathfrak{g}$ is itself a Lie group under its vector addition, it also makes sense to consider the space $\mathcal{K}(\mathfrak{g})$ and in this case $| \cdot |_{\mathfrak{g}}$ is a norm on the vector space $\mathfrak{g}$.

\subsection{Representations of exponential type}\label{subsect representations}
The space of smooth vectors $E^{\infty}$ of a representation $\pi$ of $G$ on the space $E$ carries a natural Fr\'{e}chet space topology, supplied by the seminorms
\[
p_{L}(v)=p(\pi(L)v), \qquad v\in E^{\infty}, \: p\in\operatorname*{csn}(E), \: L\in \mathfrak{D}(G),
\]
As usual, here $\pi(L)$ arises from the infinitesimal action that the representation induces on $\mathfrak{g}$, that is, explicitly $\pi(L)v=(L\gamma_{v})(e)$ where $\gamma_{v}\in C^{\infty}(G,E)$ stands for the orbit map of $v$ under $\pi$, namely, $\gamma_{v}(g)=\pi(g)v$.

In this article, we restrict our attention to representations of the following kind.

\begin{definition} \label{def: DM rep exp}A representation $\pi$ of $G$ on $E$ is said to be of \emph{exponential type} if for each $v\in E$ and $p\in\operatorname*{csn}(E)$, there is $\lambda >0$ such that
\[\sup_{g\in G}  e^{-\lambda |g|_{G}}p(\pi(g)v)<\infty.
\]
\end{definition}

Given a representation of exponential type $\pi$, we can always extend the action $\Pi$ to an algebra homomorphism $\Pi:\mathcal{K}(G)\to L(E^{\infty},E^{\infty})$ via \eqref{equ:action}. 

In the next lemma, we show that representations of exponential type in fact satisfy a certain uniform exponential growth condition.

\begin{lemma} \label{l: DM rep exp} Let $\pi$ be a representation of exponential type of $G$ on $E$. 
For each $p\in\operatorname*{csn}(E)$, one can find $C,\lambda>0$ and $q\in\operatorname*{csn}(E)$
such that
\begin{equation}
\label{eq: DM rep exp}
p(\pi(g)v)\leq C e^{\lambda |g|_{G}} q(v), \qquad \mbox{for all } v\in E,\: g\in G.
\end{equation}
\end{lemma}
\begin{proof}
Fix  $p\in\operatorname*{csn}(E)$. For $n \in \N$, we define
$$
F_n = \{ v \in E :\: p(\pi(g)v) \leq n e^{n|g|_G}, \, \mbox{for all }  g \in G\}.
$$
Then, each $F_n$ is closed and, as $\pi$ is of exponential type,
$$
E = \bigcup_{n \in \N} F_n.
$$
By the Baire category theorem, there is $n_0 \in \N$ such that $F_{n_0}^\circ \neq \emptyset$. Pick $v \in F_{n_0}$ and an open neighborhood $U$ of $0$ in $E$ such that $v+ U \subseteq F_{n_0}$. Then,
$$
U-U = (v+U) - (v+U) \subseteq F_{n_0} - F_{n_0} \subseteq F_{2n_0},
$$
which yields that $0 \in F_{2n_0}^\circ$. Hence, there are $q\in\operatorname*{csn}(E)$ and $\varepsilon >0$ such that 
$$
\{ v \in E :\: q(v) \leq \varepsilon \} \subseteq F_{2n_0}.
$$
This implies that 
$$
p(\pi(g)v)\leq \frac{2n_0}{\varepsilon} e^{2n_0 |g|_{G}} q(v), \qquad \mbox{for any } v\in E,\: g\in G.
$$

\end{proof}

\begin{remark} When $E$ is a Banach space, every representation of a Lie group is automatically of exponential type, as follows from G\aa rding's lemma \cite[Lemme 3]{G60} applied to the function $\log \|\pi(g)\|_{L(E,E)}$.
\end{remark}
\begin{remark}  Lemma \ref{l: DM rep exp} shows that a representation of exponential type on a Fr\'{e}chet space is always of $\exp(|\cdot|_{G})$-moderate growth in the sense of \cite{B-K}, a notion that essentially goes back to Casselman \cite{C}.
In the terminology from \cite{DPV21}, such representations are  projective  proto-Banach. Projective (and inductive) proto-Banach representations \cite{DPV21} generalize the notion of proto-Banach representations from \cite{Glo}.
\end{remark}

 \subsection{Weight functions}\label{Subsection: weight functions} We call an
  increasing function $\sigma \colon [0,\infty)\to [0,\infty)$ a \emph{weight function} if it satisfies the following two growth properties at $\infty$ (cf. \cite{BMT}):
\begin{itemize}
    \item[\((\alpha)\)] \(\sigma(2t) = O(\sigma(t))\).
    \item[\((\gamma)\)] \(\log t = o(\sigma(t))\).
\end{itemize}

 The importance of weight functions for us  mainly lies in the existence of entire functions satisfying certain lower and upper bounds with respect to a weight function. The following lemma is essentially\footnote{The statements of \cite[Lemma 6.1 and Proposition 6.2]{DPV21} are formulated for so-called associated functions of weight sequences, but inspection in their proofs shows they hold true for weight functions in the sense we consider here.} shown in \cite[Lemma 6.1 and Proposition 6.2]{DPV21}.

\begin{lemma}[{\cite[Lemma 6.1 and Proposition 6.2]{DPV21}}]
\label{l: DM we} Given a weight function $\sigma$, there are
an entire function $Q$ on $\C$ and a constant $a\in (0,1]$ such that, for each $n \in \N$,
\begin{equation}
\label{eq: DM we}
0<\inf_{|\Im m\: z|\leq n} e^{-a\sigma(\Re e\: z)}|Q(z)| \qquad \mbox{ and } \qquad \sup_{|\Im m\: z|\leq n}  e^{-\sigma(\Re e\: z)}|Q(z)|<\infty.
\end{equation}
\end{lemma}

 We will also need the following minorization lemma.

\begin{lemma}
\label{l: DM w}
 Let $(\tau_n)_{n\in\mathbb{N}}$ be a sequence of non-negative functions on $[0,\infty)$ such that each $\tau_n$ satisfies $\log t= o(\tau_n(t))$. There is a weight function $\sigma$ such that $\sigma(t)=o(\tau_{n}(t))$, for each $n\in\N$.
\end{lemma}
\begin{proof}
We closely follow the proof of \cite[Lemma 1.6]{BMT}. Determine inductively a sequence $(t_n)_{n \in \N}$ of positive numbers with $t_1 =e$ such that, for all $n \in \N$,
\begin{align}
\label{prop-seq}
&t_{n+1} \geq 2t_n, \qquad \log t_{n+1} \geq \max_{1 \leq k \leq n} 2^{n+1-k} \log t_k, \\
& \nonumber
 \mbox{and} \quad \log t \leq \frac{1}{n^2} \min_{1 \leq k \leq n}  \tau_k(t), \quad \mbox{for } t \geq t_n.
\end{align}
Define $\sigma \colon [0,\infty)\to [0,\infty)$ by $\sigma(t) = 0$ for $t \in [0,t_1)$ and
$$
\sigma(t) = n \log t - \sum_{k=1}^n \log t_k, \qquad \mbox{ for } t \in [t_n, t_{n+1}).
$$
Arguing exactly as in the proof of \cite[Lemma 1.6]{BMT} (with $\omega(t) = \log t$), we see that the first two properties from \eqref{prop-seq} imply that $\sigma$ is a weight function. 
We now show that $\sigma(t)=o(\tau_{n}(t))$, for each $n\in\N$.  Using the third property from \eqref{prop-seq},  we infer that
$$
\sigma(t) \leq n \log t \leq \frac{\tau_n(t)}{n}, \qquad \mbox{for } t \geq t_n.
$$
\end{proof}

\section{Factorization on the real line}\label{sec: factorization real line}
A key milestone toward our proof of Theorem \ref{th: fth} is to first establish it for the particular case $G=\R$. 

\begin{proposition}\label{prop: DM R} Let $\pi$ be a representation of exponential type of $\R$ on $E$. Then, for each bounded subset $B\subset E^{\infty}$, there are $\chi \in \mathcal{K}(\R)$ and another bounded subset $B'\subset E^{\infty}$ such that $B=\Pi(\chi)(B')$.
\end{proposition}

The remainder of this section is entirely devoted to proving Proposition \ref{prop: DM R}. Throughout this section we assume $\pi$ is as in the statement of Proposition \ref{prop: DM R} with $E^{\infty}$ its associated space of smooth vectors. We need to explain some concepts in preparation for the proof. Our considerations make crucial use of  Silva's theory \cite{Silva58} of tempered ultrahyperfunctions (also known as Silva tempered ultradistributions). We refer to \cite{DV16,HPbook} for accounts on properties, standard operations, and Fourier-Laplace analysis with ultrahyperfunctions (and distributions).

The Silva test function space $\mathcal{U}(\mathbb{C})$ consists of all those entire functions such that
\begin{equation}
\label{eq: def Silva}
\sup_{|\Im m\: z|\leq k} |\psi(z)|(1+|\Re e\:z|)^{k} < \infty,
\end{equation}
for each $k\in\mathbb{N}$. The family of norms \eqref{eq: def Silva} defines the canonical Fr\'{e}chet space structure of $\mathcal{U}(\mathbb{C})$. The space of Silva  tempered ultrahyperfunctions is defined as the topological dual $\mathcal{U}'(\mathbb{C})$ of $\mathcal{U}(\mathbb{C})$.

Given a weight function $\sigma$, we also consider the Fr\'{e}chet space
\[
\mathcal{U}_{(\sigma)}(\mathbb{C})=\Big\{\psi \in\mathcal{U}(\mathbb{C}): (\forall k\in \mathbb{N})\sup_{|\Im m\: z|\leq k} |\psi(z)|e^{k\sigma(|\Re e\: z|)} < \infty\Big\}.
\]
It is a dense subspace of the Silva space:
\begin{lemma}\label{l: dense} The subspace
$\mathcal{U}_{(\sigma)}(\mathbb{C})$ is dense in $\mathcal{U}(\mathbb{C})$.
\end{lemma}
\begin{proof}

Notice $\omega = \sigma^2$ is a weight function satisfying $\sigma = o(\omega)$. Let $Q$ be as in Lemma \ref{l: DM we} with $\omega$ as weight function. Then, $\phi = Q(0)/Q \in  \mathcal{U}_{(\sigma)}(\mathbb{C})$ and $\phi(0) = 1$. Set $\phi_n(z)=\phi(z/n)$ for $n \in \N$. Condition $(\alpha)$ (see Subsection \ref{Subsection: weight functions}) implies that $\phi_n  \in \mathcal{U}_{(\sigma)}(\mathbb{C})$ for all $n \in \N$. One readily verifies that the sequence $\psi_n= \phi_n \cdot\psi \in \mathcal{U}_{(\sigma)}(\mathbb{C})$ converges to $\psi$ in $\mathcal{U}(\mathbb{C})$.
\end{proof}
  The space  $\mathcal{K}(\R)$ of exponentially rapidly decreasing smooth functions on the real line  and the Silva space $\mathcal{U}(\mathbb{C})$ correspond to each other isomorphically under the Fourier transform. We fix here the constants in the Fourier transform as
$$
\mathcal{F}(\varphi)(\xi) = \widehat{\varphi}(\xi) = \int_{-\infty}^{\infty} \varphi(t) e^{ i \xi t} dt.
$$

The topological isomorphism of Fr\'{e}chet spaces $\mathcal{F}:\mathcal{K}(\mathbb{R})\to\mathcal{U}(\mathbb{R})$ extends to their dual spaces via the standard transposition procedure from generalized function theory. We will actually also work with $E$-valued versions of such spaces and their Fourier transforms. We define spaces of $E$-valued generalized functions as the spaces of continuous linear maps from the test function spaces into $E$, that is, we set $\mathcal{U}'(\mathbb{C},E)= L(\mathcal{U}(\mathbb{C}), E)$ and $\mathcal{K}'(\mathbb{R},E)=L(\mathcal{K}(\mathbb{R}), E)$, the spaces of $E$-valued Silva tempered ultrahyperfunctions and $E$-valued distributions of exponential type, respectively. We also work with $\mathcal{U}'_{(\sigma)}(\mathbb{C},E)=L(\mathcal{U}_{(\sigma)}(\mathbb{C}), E)$. In view of Lemma \ref{l: dense}, we might naturally regard $\mathcal{U}'(\mathbb{C},E)$ as a subspace of the larger $E$-valued ultrahyperfunction space $\mathcal{U}'_{(\sigma)}(\mathbb{C},E).$  We denote the action of an $E$-valued ultrahyperfunction or distribution $\mu$ at a corresponding test function $\psi$ by $\langle \mu,\psi \rangle \in E$. The Fourier transform of $\mu\in \mathcal{U}'(\mathbb{C}, E)$ is the $E$-valued distribution $\widehat{\mu}\in \mathcal{K}'(\mathbb{R},E)$ determined by $\langle \widehat{\mu},\varphi \rangle= \langle \mu,\widehat{\varphi }\rangle\in E$. Likewise, one defines $\mathcal{F}:\mathcal{K}'(\mathbb{R},E)\to\mathcal{U}'(\mathbb{R},E)$.

We now turn our attention to the representation $\pi$. Given a fixed $v\in E^{\infty}$, we recall that $\gamma_{v}\in C^{\infty}(\R, E)$ stands for its orbit map $\gamma_{v}(t)=\pi(t)v$. Observe that, by Lemma \ref{eq: DM rep exp}, $\gamma_{v}\in \mathcal{O}^{\exp}_{C}(\R, E)$, where the latter is the space of $E$-valued `very slowly exponentially increasing smooth functions', namely,
\begin{align*}\mathcal{O}^{\exp}_{C}(\R, E)=\{ f\in C^{\infty}(\R, E):(\forall p\in\operatorname*{csn}(E))(\exists & \lambda>0) (\forall  k\in \mathbb{N}) (p(f^{(k)} (t))=O(e^{\lambda|t|}))\}.
\end{align*}
The scalar-valued case of this and closely related spaces have been thoroughly studied in \cite{DV21}. We define a bornology on  $\mathcal{O}^{\exp}_{C}(\R, E)$: A subset $\mathcal{B}\subset \mathcal{O}^{\exp}_{C}(\R, E)$ is bounded if 
\[
(\forall p\in\operatorname*{csn}(E)) (\exists \lambda>0) (\forall  k\in \mathbb{N}) \sup_{\:f\in\mathcal{B},t\in\R}e^{-\lambda|t|}p(f^{(k)} (t))<\infty.
\]

Lemma \ref{l: DM rep exp} immediately yields the following result.

\begin{lemma}\label{l: DM bounded sets}
A subset $B\subset E^{\infty}$ is bounded if and only if $\{\gamma_{v}:\: v\in B\}\subset \mathcal{O}^{\exp}_{C}(\R, E)$ is bounded.
\end{lemma}

Our strategy to prove Proposition \ref{prop: DM R} is first to factor each representation orbit $\gamma_{v}= \chi \ast h $  with certain special elements $\chi\in\mathcal{K}(\R)$ and $h\in\mathcal{O}^{\exp}_{C}(\R, E)$, and then use this to factor $v\in E^{\infty}$. We will achieve it by working with the Fourier transforms: $\widehat{\gamma}_{v}= \widehat{\chi} \cdot \widehat{h}$. Notice $\mathcal{O}^{\exp}_{C}(\R, E)\subset \mathcal{K}'(\R, E)$, where, as usual, $f\in\mathcal{O}^{\exp}_{C}(\R, E)$ is regarded as an $E$-valued distribution of exponential type via $\langle f,\varphi \rangle= \int_{\mathbb{R}} f(t)\varphi(t)dt\in E.$ So, $\widehat{\gamma}_{v}$ makes sense as an $E$-valued tempered ultrahyperfunction. 

It is desirable to represent the elements of $\mathcal{U}'(\C, E)$ in a traceable way. First notice \cite[Remark~24.5(a), p.~279]{mv} that $E$ is isomorphic to the projective limit of a projective sequence of Banach spaces

$$ \cdots \overset{\rho^{n+1}_{n}} \longrightarrow E_n \overset{\rho^{n}_{n-1}}\longrightarrow \cdots \overset{\rho^{3}_{2}} \longrightarrow E_2\overset{\rho^{2}_{1}}\longrightarrow E_1,
$$
with continuous linear spectral maps $\rho_n^{n+1}: E_{n+1} \to E_{n}$.
We may thus simply assume that $E$ is explicitly given by
\[
E=\varprojlim_{n\in\mathbb{N}} E_n= \Big\{(v_n)_{n\in\mathbb{N}} \in \prod_{n \in \N} E_n  : \, (\forall n \in \N)( v_n = \rho_n^{n+1}(v_{n+1})) \Big\},
\]
and that $E$ is endowed with the projective limit topology, that is, the coarsest topology that makes all natural maps  $\rho_{n}: E\to E_n,  (v_k)_{k \in \N} \mapsto v_n$, continuous.  Any $\mu\in\mathcal{U}'(\mathbb{C},E)$ (and similarly for other $E$-valued spaces) can be described as $\mu=(\mu_{n})_{n\in\mathbb{N}}\in \prod_{n\in\mathbb{N}}\mathcal{U}'(\mathbb{C},E_n)$, where all $E_n$-valued tempered ultrahyperfunction components  are linked via $\mu_n=\rho^{n+1}_{n}\circ\mu_{n+1}$ (in fact, they are given by $\mu_n=\rho_n \circ \mu$).

In order to move forward, we adapt\footnote{The generalization of Silva's theory of analytic representations to Banach space-valued tempered ultrahyperfunctions is straightforward, we therefore take the Banach space-valued case for granted.} a classical idea of Silva to the $E$-valued case.  Each $\mu_n\in \mathcal{U}'(\mathbb{C},E_n)$ can \cite{DV16,HPbook}  be represented by an $E_n$-valued function $F_n$ analytic on $\mathbb{C}\setminus(\mathbb{R}+i[-\lambda_n,\lambda_n])$ and continuous on $\mathbb{C}\setminus(\mathbb{R}+i(-\lambda_n,\lambda_n))$  for some $\lambda_n>0$ and being of at most polynomial growth on each strip within such a region, namely, for each $b>\lambda_n$, there is $N=N_{n,b}$ such that 
 \[
 \sup_{\lambda_n\leq |\Im m \: z|\leq b} (1+|\Re e\: z|)^{-N}\|F_{n}(z)\|_{E_n}<\infty.
 \]
The $E_n$-valued analytic function $F_n$ represents $\mu_n$ in the sense that 
\begin{equation}
\label{l: eq an rep}
\langle \mu_n, \varphi \rangle=- \oint_{\Gamma_{b_n}}F_{n}(z)\varphi(z)dz, \qquad \varphi\in \mathcal{U}(\mathbb{C}),
\end{equation}
where, in view of Cauchy's theorem, $b_n\in [\lambda_n,\infty)$ can be arbitrarily chosen and $\Gamma_{b_n}$ stands for the counter clockwise oriented boundary of the strip $\mathbb{R}+i[-b_n,b_n]$. 
Hence, $(F_n)_{n\in\mathbb{N}}$ represents $\mu$ as
\[
\langle \mu, \varphi \rangle=\left( - \oint_{\Gamma_{b_n}}F_n(z)\varphi(z)dz\right)_{n\in\mathbb{N}}, \qquad \varphi\in \mathcal{U}(\mathbb{C}), \qquad b_n>\lambda_n.
\]

Conversely, any family of analytic functions $(F_n)_{n\in\mathbb{N}}$ having the properties we just discussed defines an $E$-valued tempered ultrahyperfunction whenever the $\mu_n$ given by \eqref{l: eq an rep} satisfy the projective constrains $\mu_{n}=\rho^{n+1}_{n}\circ\mu_{n+1}$.

We have completed all preparatory work for our proof of the bounded strong factorization theorem for $G=\R$.

\begin{proof}[Proof of Proposition \ref{prop: DM R}] 
Let $B\subset E^{\infty}$ be bounded. 
Let $\varphi_{\pm}\in C^{\infty}(\mathbb{R})$ be such that $\varphi_{-}+\varphi_{+}=1$, $\operatorname*{supp}\varphi_{-}\subseteq(-\infty,0]$,  and $\operatorname*{supp}\varphi_{+}\subseteq[-1,\infty)$. For each $v\in B$, we write its orbit map as $\gamma_{v}=(f^{v}_{n})_{n\in\mathbb{N}}$. Denote $\mathcal{B}=\{\gamma_{v}:\: v\in B\}$, a bounded subset of $\mathcal{O}_{C}^{\exp}(\mathbb{R}, E)$ in view of Lemma \ref{l: DM bounded sets}.

We describe $\mathcal{F}(\mathcal{B})$ with the aid of analytic representations. Write $e_{c}(t)=e^{c t}$. Given $n\in\mathbb{N}$, there is $\lambda_n>0$ such that, for any $b>\lambda_n$,
$$\mathcal{B}^{b}_{\pm,n}=\{e_{\mp c}\varphi_{\pm}f^{v}_{n}:\: v\in B, c\in[\lambda_n,b] \}
$$ are bounded subsets of
$$\mathcal{S}(\mathbb{R},E_n)=\Big\{\phi\in C^{\infty}(\mathbb{R}, E_n): \:(\forall k\in\mathbb{N})\sup_{t\in\mathbb{R},\: 0\leq j\leq k}(1+|t|)^{k} \|\phi^{(j)}(t)\|_{E_n} < \infty
\Big\},$$
the well-known $E_n$-valued Schwartz space of (polynomially) rapidly decreasing functions. Consequently, $\mathcal{F}(\mathcal{B}^{b}_{\pm,n})$ are bounded subsets of $\mathcal{S}(\mathbb{R},E_n)$ as well.

 The modified holomorphic Fourier-Laplace transform 
\[
F_{n}^{v}(z)=\begin{cases}\displaystyle
\int_{-\infty}^{\infty} \varphi_{+}(t)f^{v}_{n}(t)e^{iz t}d t= \mathcal{F}(e_{-\Im m \: z}\varphi_{+}f^{v}_{n})(\Re e\: z), \qquad &\Im m\:z\geq\lambda_n, \\
\displaystyle -\int_{-\infty}^{\infty} \varphi_{-}(t)f^{v}_{n}(t)e^{iz t}d t=-\mathcal{F}(e_{-\Im m \: z}\varphi_{-}f^{v}_{n})(\Re e\: z), \qquad &\Im m\:z\leq-\lambda_n,
\end{cases}
\]
then represents the tempered ultrahyperfunction $\widehat{f^{v}_{n}}\in \mathcal{U}'(\mathbb{C},E_n)$. Let 
\[
C_{n,k}= \sup_{v\in B}\sup_{|\Im m\: z|\in[\lambda_n,\lambda_n+k]} (1+|\Re e\:z|)^{k}\|F^{v}_{n}(z)\|_{E_n}, \qquad n,k \in \N.
\]
Notice that each $C_{n,k}<\infty$ due to the boundedness of $\mathcal{F}(\mathcal{B}^{k}_{\pm,n})$ in $\mathcal{S}(\mathbb{R},E_n)$. Define
\[
\tau_n (t)= \max_{k<t}( k\log (1+t) - \log C_{n,k})
\]
on $[0,\infty)$. We employ Lemma \ref{l: DM w} to produce a weight function $\sigma$ such that $\sigma(t)=o(\tau_{n}(t))$, for each $n\in\mathbb{N}$. We then apply Lemma \ref{l: DM we} to find an entire function $Q$ satisfying the bounds \eqref{eq: DM we} with respect to $\sigma$. 

The lower bound from \eqref{eq: DM we} implies that $1/Q\in\mathcal{U}(\mathbb{C})$. Hence, there is $\chi\in\mathcal{K}(\mathbb{R})$ such that $\widehat{\chi}=1/Q$.  By construction, we see that $(Q\cdot F^{v}_n)_{n\in\mathbb{N}}$ gives rise to an element $\widehat{h^{v}}\in\mathcal{U}'(\mathbb{C},E)$ as its analytic representation. The $E$-valued distributions ${h}^{v}= ({h}^{v}_{n})_{n\in\mathbb{N}}\in\mathcal{K}'(\mathbb{R},E)$ turn out to be elements of $\mathcal{O}_{C}^{\exp}(\mathbb{R}, E)$ and can indeed explicitly be computed from $(Q\cdot F^{v}_n)_{n\in\mathbb{N}}$, 
\begin{align*}
{h}^{v}_{n}(t)= &- \frac{1}{2\pi} \oint_{\Gamma_{{\lambda_n}}}Q(\zeta)F^{v}_{n}(\zeta)e^{-i t\zeta}d\zeta
\\
&
= \frac{e^{{\lambda_n} t}}{2\pi} \int_{-\infty}^{\infty}Q(\xi+i{\lambda_n})F^{v}_{n}(\xi+i{\lambda_n})e^{-i t\xi}d\xi 
\\
& \qquad \qquad  \qquad \qquad  \qquad -\frac{e^{-{\lambda_n} t}}{2\pi} \int_{-\infty}^{\infty}Q(\xi-i{\lambda_n})F^{v}_{n}(\xi-i{\lambda_n})e^{-i t\xi}d\xi.
\end{align*}
Hence, for each $n,k\in\mathbb{N}$,
\begin{align*}
\sup_{v\in B}&\sup_{t\in\mathbb{R}}\max_{0\leq j\leq k}e^{-{\lambda_n}|t|}\|(h^{v}_{n})^{(j)}(t)\|_{E_n}
\\
&\leq A_n(1+{\lambda_n})^{k} \sup_{v\in B}\sup_{\xi\in\mathbb{R}} (1+|\xi|)^{k+2}e^{\sigma(\xi)}(\| F^{v}_{n}(\xi+i{\lambda_n})\|_{E_n}+\| F^{v}_{n}(\xi-i{\lambda_n})\|_{E_n})<\infty,
\end{align*}
for certain $A_n>0$,
in view of the upper bound \eqref{eq: DM we} and the choice of $\sigma$. This shows that $\{h^{v}:\: v\in B\}$ is a bounded subset of $\mathcal{O}_{C}^{\exp}(\mathbb{R}, E)$.

Our considerations so far yield the orbit factorization $\gamma_{v}= \chi\ast h^{v}$ 
 for each $v\in B$, and in particular $v=\int_{-\infty}^{\infty}\check{\chi}(t)h^{v}(t)dt$, where we use the notation $\check{\chi}$ for its reflection about the origin, that is, $\check{\chi}(t)=\chi(-t)$. Moreover, $\{h^{v}: \: v\in B\}$ is a bounded subset of $\mathcal{O}_{C}^{\exp}(\mathbb{R}, E)$. We claim that $B'=\{h^{v}(0):\: v\in B\}\subset E^{\infty}$ is bounded and that $B=\Pi(\check{\chi})(B')$. In view of Lemma  \ref{l: DM bounded sets}, these two claims would follow at once if we show that $\gamma_{h^{v}(0)}= h^{v}$ for each $v\in B$. So, let $v\in B$. We denote by $T(t)$ the translation operators, namely, $(T(t)f)(s)=f(s-t)$. Let $t\in\mathbb{R}$ be fixed but arbitrary. We show the equality $\pi(t)h^{v}=T(
 -t)h^{v}$, which yields the wanted relation after evaluating both functions at 0 (and then letting $t$ vary). We warn the reader we perform a couple of computations below in the larger ultrahyperfunction space $\mathcal{U}'_{(\sigma)}(\mathbb{C},E)$.  In fact, it should be noticed that the multiplication of elements of $\mathcal{U}'(\mathbb{C},E)$ by $Q$ is not well-defined in general; however, $\mu\mapsto Q\cdot \mu$ is a well-defined\footnote{Indeed, $\psi\mapsto Q\cdot \psi$ is a continuous linear operator $\mathcal{U}_{(\sigma)}(\mathbb{C})\to \mathcal{U}_{(\sigma)}(\mathbb{C}) $, so that we may define the multiplication by $Q$ on  $\mathcal{U}'_{(\sigma)}(\mathbb{C},E)$ via standard transposition: $\langle Q\cdot \mu, \psi\rangle=\langle \mu, Q\cdot \psi\rangle $. } continuous linear operator 
$\mathcal{U}'_{(\sigma)}(\mathbb{C},E)\to\mathcal{U}'_{(\sigma)}(\mathbb{C},E)$. Consequently,  $\widehat{h^{v}}= Q\cdot \widehat{\gamma_{v}}$, which is a well-defined multiplication in $\mathcal{U}'_{(\sigma)}(\mathbb{C},E)$, but also happens to be an element of $\mathcal{U}'(\mathbb{C},E)$. We then have
\[
\pi(t)h^{v}= \pi(t) \mathcal{F}^{-1}(Q\cdot \mathcal{F}({\gamma_{v}}) )= \mathcal{F}^{-1}(Q\cdot\mathcal{F}( \pi(t){\gamma_{v}}) )=\mathcal{F}^{-1}(Q\cdot\mathcal{F}( T(-t){\gamma_{v}}) )=T(-t)h^{v}.
\]

This completes the proof of Proposition \ref{prop: DM R}.

\end{proof}

\section{Exponential coordinates of second kind and the algebra $\mathcal{K}(G)$}
\label{sec exp coordinates}

The second main ingredient in our proof of Theorem \ref{th: fth} will be showing that the pullback
of $\mathcal{K}(G)$ by suitable global exponential coordinates of the second kind \cite{F-LBook} contains the corresponding algebra $\mathcal{K}(\mathfrak{g})$ of exponentially rapidly decreasing functions on the Lie algebra $\mathfrak{g}$.  The latter is the content of Theorem \ref{pushforward} below. We study in this section features  of the special coordinates making such a result possible. 

Since $\mathfrak{g}$ is solvable, the Lie subalgebra $\mathfrak{n}:=[\mathfrak{g},\mathfrak{g}]$ is nilpotent. It is shown in \cite[p. 135]{F-LBook} that one can select a so-called coexponential basis to $\mathfrak{n}$ in $\mathfrak{g}$, that is, a  basis $\{A_1,\dots,A_n\}$ of a subspace complementary to $\mathfrak{n}$ in $\mathfrak{g}$ such that $t_1A_1+\dots +t_nA_n+X\to \exp(t_1A_1)\dots \exp(t_nA_n)\exp(X)$ becomes a diffeomorphism from the Lie algebra $\mathfrak{g}= \mathbb{R}A_1\oplus \dots \oplus\mathbb{R}A_n\oplus\mathfrak{n}$ onto $G$. If we now pick a Malcev basis \cite{C-GBook,F-LBook} $\{N_1,\dots,N_m\}$ for the nilpotent Lie algebra $\mathfrak{n}$, we obtain the following result at once.

\begin{lemma}\label{Malcev}
    The map $\Phi:\mathfrak{g}\to G$ given by
    $$t_1A_1+\dots+t_nA_n+s_1N_1+\dots+s_mN_m\mapsto \exp(t_1A_1)\dots\exp(t_nA_n)\exp(s_1N_1)\dots\exp(s_mN_m)$$
    is a diffeomorphism of smooth manifolds.
\end{lemma}

We fix from now on the basis $\{A_1,\dots, A_n,N_1,\dots, N_{m}\}$ giving rise to the diffeomorphism $\Phi$ as in Lemma \ref{Malcev} and demonstrate a number of properties of $\Phi$. We start with a simple growth bound for $\Phi$. Recall $|\cdot|_{G}$ and $|\cdot|_{\mathfrak{g}}$ were introduced in Subsection \ref{subsec exp decreasing}.
\begin{lemma}\label{Garding}
    There exist constants $b,B>0$ such that, for any $X\in \mathfrak{g}$,
  \begin{equation}\label{bound phi}
  |\Phi(X)|_G\leq b|X|_{\mathfrak{g}}+B.
  \end{equation}
\end{lemma}
\begin{proof} Rename $X_k:=A_k$, for $1\leq k\leq n$, and $X_{n+k}=N_k$, for $1\leq k\leq m$. Using \eqref{eq submult}, we obtain
$|\Phi(\sum_{k=1}^{n+m}t_kX_k)|_{G}\leq \sum_{k=1}^{n+m} |\exp(t_k X_j)|_{G}$. By the same reason, each continuous  function $ |\exp(t X_{j})|$ is subadditive in the variable $t$ and, by a classical lemma due to Beurling \cite{B38} (generalized to Lie groups by G\aa rding \cite[Lemme 3]{G60}), they are all bounded by a linear function of $t$. We thus conclude that, for some $B,b'>0$, the inequality $|\Phi(X)|_{G}\leq B+b' \sum_{k=1}^{n+m} |t_k|$ holds for any $X=\sum_{k=1}^{n+m} t_k X_k$. Since $|X|_{1}=\sum_{k=1}^{n+m}|t_k|$ is also a norm and all norms on $\mathfrak{g}$ are equivalent, we obtain that the bound \eqref{bound phi} must hold for some $b>0$.
\end{proof}

Denote by $\mathcal{O}_{M}^{\exp}(\mathfrak{g})$ the algebra of exponentially bounded functions on $\mathfrak{g}$ with exponentially bounded
derivatives\footnote{This function space is the exponential analog of the Schwartz space $\mathcal{O}_{M}(\mathfrak{g})$ of multipliers \cite{SBook} for the Schwartz class $\mathcal{S}(\mathfrak{g})$. }. That is,
    
       $$\mathcal{O}_{M}^{\exp}(\mathfrak{g})=\left\{f\in C^{\infty}(\mathfrak{g}):(\forall P\in\mathfrak{D}(\mathfrak{g}))( \exists \lambda>0)\sup_{X \in \mathfrak{g}} |P(f)(X)|e^{-\lambda|X|_{\mathfrak{g}}} < \infty\right\}$$
and let $\mathfrak{D}_{\exp}(\mathfrak{g})$ be the algebra of differential operators generated by $\mathcal{O}_{M}^{\exp}(\mathfrak{g})$ and $\mathfrak{D}(\mathfrak{g})$. By definition, the algebra $\mathcal{O}_{M}^{\exp}(\mathfrak{g})$ is $\mathfrak{D}(\mathfrak{g})$-stable; therefore, any element of $\mathfrak{D}_{\exp}(\mathfrak{g})$ can be written as a finite sum $\sum f_kP_k$, where each $f_k\in\mathcal{O}_{M}^{\exp}(\mathfrak{g})$ and $P_k\in\mathfrak{D}(\mathfrak{g})$. In other words, the algebra $\mathfrak{D}_{\exp}(\mathfrak{g})$ can be thought of as the algebra of linear partial differential operators with smooth coefficients that are exponentially bounded, along with all their derivatives.

The next proposition is crucial for us. We have, 

\begin{proposition}\label{pullback}
    For any $L\in \mathfrak{D}(G)$, we have $\Phi^{\ast}(L)\in\mathfrak{D}_{\exp}(\mathfrak{g})$ and $\Phi^{\ast}(\iota_{\ast}L)\in\mathfrak{D}_{\exp}(\mathfrak{g})$.
\end{proposition}
\begin{proof} Note that $\Phi^{\ast}(L_1L_2)= \Phi^{\ast}(L_1)\Phi^{\ast}(L_2)$. In view of the classical Poincar\'{e}-Birkhoff-Witt theorem \cite[Theorem~1.2.7, p.~23]{F-LBook} and the well-known fact that one may identify $\mathfrak{D}(G)$ with the universal enveloping algebra of $\mathfrak{g}$, it suffices to verify that $\Phi^{\ast}(X)\in \mathfrak{D}_{\exp}(\mathfrak{g})$ and $\Phi^{\ast}(\iota_{\ast}X)\in\mathfrak{D}_{\exp}(\mathfrak{g})$ 
for all left-invariant vector fields $X\in \mathfrak{g}$. Both proofs are completely similar, we shall therefore only consider here the pullback of a left-invariant vector field $X\in\mathfrak{g}$ and leave the case of $\iota_{\ast}X$ to the reader.
We endow $\mathfrak{g}$ with the coordinates $ (t_1,\dots,t_n,s_1,\dots,s_m)$ arising from the basis $\{A_1,\dots,A_n,N_1,\dots,N_m\}$. For the remainder of the proof, we may  identify $\mathfrak{g}$ with $\mathbb{R}^{n+m}$ via these coordinates. Under this identification, the algebra $\mathfrak{D}(\mathfrak{g})$ simply corresponds to the usual algebra of differential operators with constant coefficients $\mathbb{R}[\partial_{t_1},\dots,\partial_{t_n},\partial_{s_1},\dots,\partial_{s_m}]$. For any $X\in \mathfrak{g}$, there exist smooth functions $\phi_i^X$ and $\psi_j^X$ in $C^{\infty}(\mathfrak{g})$ such that
    $$\Phi^*X=\sum_{i=1}^{n}\phi^X_i\partial_{t_i}+\sum_{j=1}^{m}\psi^X_j\partial_{s_j}$$
    where, since $X$ is a left-invariant vector field, we can compute the left-hand side via
    $$\Phi^{\ast}X(t_1,\ldots,t_n,s_1,\ldots,s_n)=\left.\frac{d}{dr}\right|_{r=0}\Phi^{-1}(\Phi(t_1,\dots,t_n,s_1,\dots,s_m)\cdot\exp(r X)),$$    
    and the functions $\phi^X_i$ and $\psi^X_j$ correspond to the components of the $\R^{n+m}$-valued function that appear on the right-hand side of the previous formula. We only need to prove each $\phi^X_i$ and $\gamma^X_j$ belongs to $\mathcal{O}_{M}^{\exp}(\mathfrak{g})$, and it is enough to do so for $X$ equal to each of the elements in $\{A_1,\dots,A_n,N_1,\dots,N_m\}$.
    
    If $X=N_j$ (under our identification, $X=e_{n+j}$), using that $\{N_1,\dots,N_m\}$ is a Malcev basis for $\mathfrak{n}$, we get
    \begin{align*}
        &\Phi^{-1}\left(\Phi(t_1,\dots,t_n,s_1,\dots,s_m)\cdot\exp(0,\dots,0,r,0\dots,0)\right)\\
        =&\Phi^{-1}\left(\exp(t_1A_1)\cdots\exp(t_nA_n)\exp(s_1N_1)\cdots\exp(s_mN_m)\exp(rN_j)\right)\\
        =&\Phi^{-1}\left(\exp(t_1A_1)\cdots\exp(s_{j-1}N_{j-1})\exp(p_j^j(s_j,\dots,s_m,r)N_j)\cdots\exp(p_m^j(s_j,\dots,s_m,r)N_m)\right)\\
        =&(t_1,\dots,s_{j-1},p_{j}^j(s_j,\dots,s_m,r),\dots,p_m^j(s_j,\dots,s_m,r)),
    \end{align*}
    where the functions $p_k^j$ are polynomials in
    $m-j+2$  the variables  $s_j,\ldots,s_m,r$, and moreover $p_j^j=s_j+r$. From this we obtain
    $$\Phi^{\ast}N_j=(0,\dots,0,1, \left.\frac{d}{dr}\right|_{r=0}p_{j+1}^j ,\dots,\left.\frac{d}{dr}\right|_{r=0}p_m^j)$$
     Each of the functions $\frac{d}{dr}|_{r=0}p_k^j$ is again a polynomial, and therefore an element of $\mathcal{O}_{M}^{\exp}(\mathfrak{g})$. We conclude $\Phi^{\ast}N_j\in\mathfrak{D}_{\exp}(\mathfrak{g})$.\par
    If $X=A_i$ (that is, $X=e_i$), we can write
    \begin{align*}
        &\Phi^{-1}\left(\Phi(t_1,\dots,t_n,s_1,\dots,s_m)\cdot\exp(0,\dots,0,r,0\dots,0)\right)\\
        =&\Phi^{-1}\left(\left(\prod_{k=1}^{n}\exp(t_kA_k)\prod_{l=1}^{m}\exp(s_lN_l)\right)\exp(rA_i)\right)\\
        =&\Phi^{-1}\left(\left(\prod_{k=1}^{i-1}\exp(t_kA_k)\right)\exp((t_i+r)A_i)\exp(-rA_i)g\exp(rA_i)\right)
    \end{align*}
    where $g=\prod_{k=i+1}^n\exp(t_kA_k)\prod_{l=1}^m\exp(s_lN_l)$. From the definition of the adjoint action,  
    we see that
    $$
    \exp(-rA_i)g=g\exp(\text{Ad}_{g^{-1}}(-rA_i))
    $$
    and, from an approximative version of the Baker-Campbell-Haussdorff formula (see for instance \cite[Theorem 1.1.19(1)]{F-LBook}),
    \begin{align*}
    \exp(\text{Ad}_{g^{-1}}(-rA_i))\exp(rA_i)&=\exp(\text{Ad}_{g^{-1}}(-rA_i)+rA_i+o(r))\\
    &=\exp((\text{Ad}_{g^{-1}}-\text{Id})(-rA_i)+o(r). 
    \end{align*}
    As an operator on $\mathfrak{g}$, $\text{Ad}_{g^{-1}}$ is given by
    \begin{align*}
        \text{Ad}_{g^{-1}}&=\text{Ad}_{\exp(-s_mN_m)}\circ\dots \circ\text{Ad}_{\exp(-t_{i+1}A_{i+1})}\\
        &=\exp(\text{ad}_{-s_mN_m})\circ\dots\circ \exp(\text{ad}_{-t_{i+1}A_{i+1}})\\
        &=\exp(-s_m\text{ad}_{N_m})\circ\dots\circ \exp(-t_{i+1}\text{ad}_{A_{i+1}}).
    \end{align*}
    Let $\|\cdot\|_{\text{op}}$ be the operator norm on the space $L(\mathfrak{g},\mathfrak{g})$. For any $X\in\mathfrak{g}$ and $t\in\mathbb{R}$, we have
    \begin{align*}
        \|\exp(-t\  \text{ad}_{X})\|_{\text{op}}&\leq \exp\left(|t|\ \|\text{ad}_{X}\|_{\text{op}}\right),
    \end{align*}
    which means all the functions that appear in the matrices that represent the operators $\exp(-s_l\text{ad}_{N_l})$ and $\exp(-t_k\text{ad}_{A_k})$, in the basis we are working with, are exponentially bounded. Since we also have 
    $$\frac{\partial}{\partial_{s_l}}\exp(-s_l\text{ad}_{N_l})=- \text{ad}_{N_l}\circ \exp(-s_l\text{ad}_{N_l}),$$
    then all these functions not only are exponentially bounded but actually belong to $\mathcal{O}_{M}^{\exp}(\mathfrak{g})$. The same argument holds for the functions in each of the variables $t_k$ appearing in the matrices that represent the operators $\exp(-t_k\text{ad}_{A_k})$. Since $\mathcal{O}_{M}^{\exp}(\mathfrak{g})$ is an algebra, we conclude the functions that appear in the matrix form of $\text{Ad}_{g^{-1}}$ all belong to $\mathcal{O}_{M}^{\exp}(\mathfrak{g})$. Moreover, the image of $\text{Ad}_{g^{-1}}-\text{Id}$ is contained in $\mathfrak{n} = [\mathfrak{g}, \mathfrak{g}]$. We obtain:
    $$\exp((\text{Ad}_{g^{-1}}-\text{Id})(-rA_i)+o(r))=\exp\left(r\sum_{l=1}^mf_{l,i}(t_{i+1},\dots,s_m)N_l+o(r)\right),$$
    where the $f_l$ are functions that belong to $\mathcal{O}_{M}^{\exp}(\mathfrak{g})$. If we put all this together we conclude that
    $$\Phi^{\ast}A_i=(0,\dots,1,\dots,0,q_{1,i}(f_{1,i},\dots,f_{m,i}),\dots, q_{m,i}(f_{1,i},\dots,f_{m,i})),$$
    where the $q_j$ are polynomials. Therefore, each entry is again a function in $\mathcal{O}_{M}^{\exp}(\mathfrak{g})$. We conclude $\Phi^{\ast}A_i\in\mathfrak{D}_{\exp}(\mathfrak{g})$, as desired.
\end{proof}

\begin{remark}
    In fact, our construction above proves that the functions $q_{l,i}(f_{1,i},\dots,f_{m,i})$ are analytic. Therefore, the pullbacks $\Phi^{\ast}(\mathfrak{D}(G))$ and $\Phi^{\ast}(\iota_{\ast}\mathfrak{D}(G))$ are differential operators with analytic coefficients that are also exponentially bounded, along with all their derivatives. We will not make use this fact in the remainder of the article. 
\end{remark}

As a consequence of the explicit formulas written down in the proof of the previous lemma, we obtain the following proposition which will be used in the next section.
\begin{proposition}\label{Haar}
    Let $\,d\mu$ be a Haar measure on $\mathfrak{g}$, then $\Phi_{\ast}(\,d\mu)$ is also a Haar measure on $G$.
\end{proposition}
\begin{proof}
Note that $\{A_1,\dots,N_m\}$ is a basis of the tangent bundle of $G$. As usual, choose $\{\partial_{t_1},\dots,\partial_{s_m}\}$ to be the corresponding basis for the tangent bundle of $\mathfrak{g}$. The formulas of the previous lemma imply that the change-of-basis matrix relating the bases $\{\partial_{t_1},\dots,\partial_{s_m}\}$ and $\{\Phi^{\ast}A_1,\dots,\Phi^{\ast}N_m\}$ is given by
 $$M:=\left(\begin{array}{cc}
        \text{Id} & 0 \\
        Q & T
    \end{array}\right)$$
    where the entries of $Q$ are given by the functions $q_{l,j}(f_{1,j},\dots,f_{m,j})$ that appear in the proof of Proposition \ref{pullback} and $T$ is a lower triangular matrix with $1$'s in its diagonal. 
    Therefore, its determinant is equal to $1$.
    
    Let $\{A_1^{\ast},\dots,N_m^{\ast}\}$ be invariant $1$-forms dual to $\{A_1,\dots,N_m\}$. The change-of-basis matrix between this basis of the cotangent bundle of $G$ and $\{\Phi_{\ast}(dt_1),\dots,\Phi_{\ast}(ds_m)\}$ is given by the transpose of $M$, which thus has determinant 1. Hence
    $$\Phi_{\ast}(\,dt_1\wedge\dots\wedge\,ds_m)=A_1^{\ast}\wedge\dots\wedge N_m^{\ast}.$$
Therefore, the pushforward of a Haar measure is again a Haar measure.
\end{proof}

We can now prove the main result of this section. \begin{theorem}\label{pushforward} We have
    $$\Phi_{\ast}(\mathcal{K}(\mathfrak{g}))\subset \mathcal{K}(G).$$
\end{theorem}
\begin{proof}
    One readily verifies that the action of differential operators on functions turns $\mathcal{K}(\mathfrak{g})$ into a $\mathfrak{D}_{\exp}(\mathfrak{g})$-module, that is, $P\varphi \in\mathcal{K}(\mathfrak{g})$ whenever 
 $\varphi\in \mathcal{K}(\mathfrak{g})$ and $P\in\mathfrak{D}_{\exp}(\mathfrak{g})$. 
 
    Let $f\in \mathcal{K}(\mathfrak{g})$, $L\in\mathfrak{D}(G)$, and fix a real number $\lambda>0$. We have
    \begin{align*}
        \sup_{g\in G}e^{\lambda|g|_G}(|&L(\Phi_*(f))(g)|+|(\iota_{\ast}L)(\Phi_*(f))(g)|)
        \\
        &=\sup_{X\in\mathfrak{g}}e^{\lambda |\Phi(X)|_G}(|(\Phi^*L)(f)(X)|+|(\Phi^*\iota_{\ast}L)(f)(X)|)\\
        &\leq e^{\lambda B}\sup_{X\in\mathfrak{g}}e^{\lambda b|X|_{\mathfrak{g}}}(|(\Phi^*L)(f)(X)|+|(\Phi^*\iota_{\ast}L)(f)(X)|)\\
        &<\infty,
    \end{align*}
 where $B$ and $b$ are the constants from Lemma \ref{Garding}, and, in the third line, we used that $(\Phi^*L)(f)\in \mathcal{K}(\mathfrak{g})$ and $(\Phi^*\iota_{\ast}L)(f)\in \mathcal{K}(\mathfrak{g})$ in view of Proposition \ref{pullback} and the $\mathfrak{D}_{\exp}(\mathfrak{g})$-stability of $\mathcal{K}(\mathfrak{g})$. The conclusion is thus that $\Phi_{\ast}f\in \mathcal{K}(G)$.
\end{proof}

\section{Factorization on exponential solvable Lie groups}\label{sec factorization exponential}
Armed with Proposition \ref{prop: DM R}, Proposition \ref{pullback}, and Theorem \ref{pushforward}, we can now deduce Theorem \ref{th: fth} via an iteration argument that goes back to Dixmier and Malliavin \cite{D-M}.

\begin{proof}[Proof of Theorem \ref{th: fth}]
Restricting each $\pi(g)$ to the invariant space $E^{\infty}$, we may just assume that the Fr\'{e}chet space satisfies $E^{\infty}=E$. 

Given $X\in\mathfrak{g}$, we can induce a representation $\pi_{X}$ of $\mathbb{R}$ on $E$ via the $L(E^{\infty},E^{\infty})$-valued one-parameter group
\[
\pi_{X}(t)= \pi(\exp(tX)), \qquad t\in\mathbb{R}.
\]
Since every vector is smooth for $\pi$, we get $E^\infty=E^{\infty}_{\pi_X}$, where the latter denotes the space of smooth vectors of $\pi_{X}$. The equality $E^\infty=E^{\infty}_{\pi_X}$ not only holds as sets, but also topologically, thanks to the open mapping theorem for Fr\'{e}chet spaces. 

Let us consider the diffeomorphism $\Phi$ from Lemma \ref{Malcev} associated with the basis $\{A_1,\dots, A_n,N_1,\dots, N_m\}$ of $\mathfrak{g}$ and write $d\mu$ for the Haar measure of $\mathfrak{g}$ given as the volume form $dt_1\wedge\dots\wedge\,ds_m$ in the coordinates with respect to this fixed basis. Note that $\Phi_{\ast}(d\mu)=c dg$ for some $c>0$
in view of Proposition \ref{Haar}.  Given a bounded subset $B\subset E^{\infty}=E^{\infty}_{\pi_{N_{m}}}$, by Proposition \ref{prop: DM R}, we can find a bound subset $B_{m+n}\subset E^{\infty}_{\pi_{N_{m}}}=E^{\infty}$ and a function $\varphi_{n+m}\in\mathcal{K}(\mathbb{R})$ such that $B=\Pi(\varphi_{m+n})(B_{m+n})$. Iterating this procedure $m+n-1$ more times, we find a bounded subset $B'\subset E^{\infty}_{\pi_{A_{1}}}=E^{\infty} $ and functions $\varphi_1,\dots,\varphi_{m+n-1}\in\mathcal{K}(\mathbb{R})$ such that
\begin{equation}
\label{eq facto iterative}
 B=\Pi_{A_1}(\varphi_{1}) \cdots \Pi_{A_{n}}(\varphi_n) \Pi_{N_{1}}(\varphi_{n+1}) \cdots \Pi_{N_{m}}(\varphi_{n+m})(B').
\end{equation}
Let now \[
\varphi\left(\sum_{j=1}^{n} t_{j}A_j + \sum_{k=1}^{m} s_k N_k\right)=\varphi_{1}(t_1)\cdots \varphi_{n}(t_n)\varphi_{n+1}(s_1) \cdots \varphi_{n+m}(s_m).\]
Clearly, $\varphi\in \mathcal{K}(\mathfrak{g})$. If $v\in B$, then \eqref{eq facto iterative} tells us that there is some $v'\in B'$ such that
\begin{align*}
v&= \int_{\mathbb{R}^{n+m}}\varphi_{1}(t_1)\cdots \varphi_{n+m}(s_m) \pi(\exp(t_1A_1)\cdots \exp(s_m N_m))v' dt_1\cdots ds_m\\
&=
 \int_{\mathfrak{g}}\varphi(X) \pi(\Phi(X))v'  d\mu(X)
 \\
 &
 = c \int_{G} [\Phi_{\ast}(\varphi)(g)] \pi(g)v'  dg
 \\
 &
 = \Pi[c\Phi_{\ast}(\varphi)] v'.
\end{align*}
Consequently, $B=\Pi(\chi)(B')$ with $\chi= c\Phi_{\ast}(\varphi)$ and $\chi\in \mathcal{K}(G)$ by Theorem \ref{pushforward}, which concludes the proof of the factorization theorem.
\end{proof}


\begin{thebibliography}{99}

\bibitem{B-K} J.~ Bernstein, B.~Kr\"{o}tz, \emph{Smooth Fréchet globalizations of Harish-Chandra modules,
}  Israel J. Math. \textbf{199}  (2014), 45--111.

\bibitem{B38} A.~Beurling, \emph{Sur les int\'{e}grales de Fourier absolument convergentes et leur application \`{a} une transformation fonctionelle,} in: IX Congr. Math. Scand., pp. 345--366, Helsingfors, 1938.
 
    \bibitem{BMT} 
   
    R.~W.~Braun, R.~Meise, B.~A.~Taylor, \emph{Ultradifferentiable functions and Fourier analysis}, Results Math. \textbf{17} (1990), 206--237.

\bibitem{C} W.~Casselman, \emph{Canonical extensions of Harish-Chandra modules to representations of
$G$,} Canadian J. Math. \textbf{41} (1989), 385--438.


\bibitem{C-GBook}L.~J.~Corwin, F.~P.~Greenleaf, \emph{Representations of nilpotent Lie groups and their
applications. Part I,} Cambridge University Press,  Cambridge, 1990.

    \bibitem{D-H-V25} A.~Debrouwere, M.~Huttener, J.~Vindas, \emph{Strong factorization of ultradifferentiable vectors associated with compact Lie group representations,} Int. Math. Res. Not. IMRN \textbf{2025} (2025), Article rnaf157.
   
    \bibitem{DPV21}
    A.~Debrouwere, B.~Prangoski, J.~Vindas, \emph{Factorization in Denjoy-Carleman classes associated to representations of \((\mathbb{R}^d,+)\)}, J. Funct. Anal. \textbf{280} (2021), Article 108831.

\bibitem{DV16} A.~Debrouwere, J.~Vindas, \emph{On the non-triviality of certain spaces of analytic functions. Hyperfunctions and ultrahyperfunctions of fast growth,} Rev. R. Acad.  Cienc. Exactas F\'{i}s. Nat. Ser. A. Math. RACSAM \textbf{112} (2018), 473--508.

\bibitem{DV21}A.~Debrouwere, J.~Vindas, \emph{Topological properties of convolutor spaces via the short time Fourier transform,} Trans. Amer. Math. Soc. \textbf{374} (2021), 829--861.

\bibitem{HPbook} R.~F.~Hoskins, J.~Sousa~Pinto, \emph{Theories of generalised functions. Distributions, ultradistributions
and other generalised functions,} Horwood Publishing Limited, Chichester, 2005.

    \bibitem{D-M} 
    J.~Dixmier, P.~Malliavin, \emph{Factorisations de fonctions et de vecteurs ind\'{e}finiment diff\'{e}rentiables}, Bull. Sci. Math. \textbf{102} (1978), 307--330.

    \bibitem{ehrenpreis}
    L.~Ehrenpreis, \textit{Solution of some problems of division. IV. Invertible and elliptic operators}, Amer. J. Math. \textbf{82} (1960), 522--588.

\bibitem{F-LBook}H.~Fujiwara, J.~Ludwig, \emph{Harmonic analysis on exponential solvable Lie groups,}  Springer, Tokyo, 2015.

\bibitem{G60} L.~G\aa rding, \emph{Vecteurs analytiques dans les repr\'{e}sentations des groupes de Lie,} Bull. Soc. Math.
France \textbf{88} (1960), 73--93.



    \bibitem{G-K-L} 
    H.~Gimperlein, B.~Kr\"{o}tz, C.~Lienau, \emph{Analytic factorization of Lie group representations}, J. Funct. Anal. \textbf{262} (2012), 667--681.



    \bibitem{Glo}
    H.~Gl\"{o}ckner, \emph{Continuity of \(LF\)-algebra representations associated to representations of Lie
    groups,} Kyoto J. Math. \textbf{53} (2013), 567--595.
    
    \bibitem{H} R.~Howe, \emph{On a connection between nilpotent groups and oscillatory integrals associated to singularities},  Pacific J. Math.\textbf{73} (1977), 329--363.

\bibitem{mv} R.~Meise, D.~Vogt, \emph{Introduction to functional analysis,} Oxford University Press, New York, 1997.

    \bibitem{rst} 
    L.~A.~Rubel, W.~A.~Squires, B.~A.~Taylor, \textit{Irreducibility of certain entire functions with applications to harmonic analysis}, Ann. of Math. \textbf{108} (1978), 553--567.

\bibitem{SBook}L.~Schwartz, \emph{Th\'{e}orie des distributions,} Hermann, Paris, 1966.

\bibitem{Silva58} J.~Sebasti\~{a}o~e~Silva, \emph{Les fonctions analytiques comme ultra-distributions dans le calcul
op\'{e}rationnel,} Math. Ann. \textbf{136} (1958), 58--96.


    \bibitem{yul} 
    R.~S.~Yulmukhametov, \textit{Solution of the L. Ehrenpreis problem on factorization}, Mat. Sb. \textbf{190} (1999), 123--157; translation in Sb. Math. \textbf{190} (1999), 597--629.

\end{thebibliography}
\end{document}